\documentclass[12pt,twoside]{article}
\usepackage{amssymb,amsmath,amsthm,latexsym}
\usepackage{amsfonts}
\usepackage{enumerate}

\newtheorem{thm}{Theorem}[section]
\newtheorem{conj}[thm]{Conjecture}

\newtheorem{lem} [thm]{Lemma}

\theoremstyle{definition} % Definitions, examples, remarks, algorithms should be in roman type, not italic.

\newtheorem{defn}[thm]{Definition}
\numberwithin{equation}{section}

\begin{document}
	\label{'ubf'}
	\setcounter{page}{1} %Put here the starting page number
	
	\markboth {\hspace*{-9mm} \centerline{\footnotesize \sc
			% Put here the left page top label
			On Signed Distance in Product of Signed Graphs}
	}
	{ \centerline {\footnotesize \sc
			%put here the author's name
			Shijin T V, Soorya P, Shahul Hameed K, Germina K A} \hspace*{-9mm}
	}
	\begin{center}
		{
			{\huge \textbf{On Signed Distance in \\ Product of Signed Graphs
					% Put the title of the paper here
				}
			}
			
			\bigskip
			Shijin T V \footnote{\small Department of Mathematics, Central University of Kerala, Kasaragod - 671316,\ Kerela,\ India.\ Email: shijintv11@gmail.com}
			Soorya P \footnote{\small  Department of Mathematics, Central University of Kerala, Kasaragod - 671316,\ Kerela,\ India.\ Email: sooryap2017@gmail.com}
			Shahul Hameed K \footnote{\small  Department of
			Mathematics, K M M Government\ Women's\ College, Kannur - 670004,\ Kerala,  \ India.  E-mail: shabrennen@gmail.com}
			Germina K A \footnote{\small  Department of Mathematics, Central University of Kerala, Kasaragod - 671316,\ Kerela,\ India.\ Email: srgerminaka@gmail.com}
	%		Thomas Zaslavsky\footnote{Department of Mathematical Sciences, Binghamton University (SUNY), Binghamton, NY 13902-6000, U.S.A.  \textbf{e-mail:} {\tt zaslav@math.binghamton.edu}}
			
		}
\end{center}
\newcommand\spec{\operatorname{Spec}}
\thispagestyle{empty}
\begin{abstract}
A signed graph is an ordered pair $\Sigma=(G,\sigma),$ where $G=(V,E)$ is the underlying graph of $\Sigma$ with a signature function $\sigma:E\rightarrow \{1,-1\}$. 
Notion of signed distance and distance compatible signed graphs are introduced in \cite{sdist}. In this article, first we characterize the distance compatibilty in the case of a connected signed graph and discussed the distance compatibility criterion for the cartesian product, lexicographic product and tensor product of signed graphs. We also deal with the distance matrix of the cartesian product, lexicographic product and tensor product of signed graphs in terms of the distance matrix of the factor graphs. 
\end{abstract}


%\newpage
%\tableofcontents
%\setcounter{tocdepth}{3}
\section{Introduction}
Unless otherwise mentioned, in this paper we consider only simple, finite and connected graphs and signed graphs. A signed graph $\Sigma=(G,\sigma)$ is an underlying graph $G=(V,E)$ with a signature function $\sigma:E\rightarrow \{1,-1\}$. We call a signed graph $\Sigma$ as  balanced (or cycle balanced) if all of its cycles are positive, where the sign of a cycle is the product of the sign of its edges. The notion of signed distance for signed graphs and distance compatible signed graphs are introduced in ~\cite{sdist}. In this paper,  first we characterize distance compatible signed graphs and analyse these concepts for the product of signed graphs such as cartesian, lexicographic and tensor products. We get explicit formulae for distance matrices for these signed graph products.  

To begin with we borrow some definitions  and notations from ~\cite{sdist}. The sign of a path $P$ in $\Sigma$ is defined as $\sigma(P)=\prod_{e\in E(P)} \sigma(e)$.

Let $u$ and $v$ be two vertices in $\Sigma$. If $u$ and $v$ are adjacent then, we will deonte the adjacency by $u\sim v$. Let $d(u,v)$ denote the usual distance between $u$ and $v$ in the underlying graph and let $\mathcal{P}_{(u,v)}$ denote the collection of all shortest paths $P_{(u,v)}$ between them.  Then, the two types of (signed) distance between $u$ and $v$ in a signed graph is defined in \cite{sdist} as:

$d_{\max}(u,v) = \sigma_{\max}(u,v) d(u,v)=\max\{\sigma(P_{(u,v)}): P_{(u,v)} \in \mathcal{P}_{(u,v)} \}d(u,v)$ and 

$d_{\min}(u,v) = \sigma_{\min}(u,v) d(u,v)=\min\{\sigma(P_{(u,v)}): P_{(u,v)} \in \mathcal{P}_{(u,v)} \}d(u,v).$

Two vertices $u$ and $v$ in a signed graph $\Sigma$ are said to be \emph{distance-compatible} (briefly, \emph{compatible}) \cite{sdist} if $d_{\min}{(u,v)}=d_{\max}{(u,v)}$. For any $u\sim v,$ $d_{\min}{(u,v)}=d_{\max}{(u,v)}.$ Also a signed graph $\Sigma$ is said to be (distance-) compatible if every pair of vertices is compatible and $\Sigma$ is incompatible, otherwise.

Corresponding to these distances, there are two types of distance matrices in a signed graph as given below.
\begin{defn}\cite{sdist} [Signed distance matrices] 

	\par(D1)  $D^{\max}(\Sigma)=(d_{\max}(u,v))_{n\times n}$.
	\par(D2)  $D^{\min}(\Sigma)=(d_{\min}(u,v))_{n\times n}$.
	
\end{defn}
We adopt the construction of a signed complete graph described in~\cite{sdist} obtained from the signed distance matrices $D^{\max}$ and $D^{\min}$ and call that signed complete graph as the associated signed complete graphs associated with $\Sigma$.
\begin{defn}\cite{sdist}
	The associated signed complete graph $K^{D^{\max}}(\Sigma)$ with respect to $D^{\max}(\Sigma)$ is obtained by  joining the non-adjacent vertices of $\Sigma$ with edges having signs 
	\begin{equation*}
	\sigma(uv)= \sigma_{\max}(uv)
	\end{equation*}
	
The associated signed complete graph $K^{D^{\min}}(\Sigma)$ with respect to $D^{\min}(\Sigma)$ is obtained by joining the non-adjacent vertices of $\Sigma$ with edges having signs 
	\begin{equation*}
	\sigma(uv)= \sigma_{\min}(uv)
	\end{equation*}
\end{defn}
Whenever, $D^{\max}=D^{\min}=D^\pm$, say, the associated signed complete graph of $\Sigma$ is denoted by $K^{D^\pm}$
\section{Characterization of compatible signed graphs}\label{mainsection}
In this section we characterize compatibile signed graphs. The following lemma which is used for proving the characterization is significant in its own right.
\begin{lem}\label{L1}
	
	Let $u$ and $v$ be incompatible pair of vertices with least distance in a $2$-connected non-geodetic signed graph. Then, there will be two internally disjoint shortest paths from $u$ to $v$ of opposite signs.
\end{lem}

\begin{proof}
	Suppose that $\Sigma=(G,\sigma)$ is a $2$-connected non-geodetic signed graph. Let $u$ and $v$ be an incompatible pair of vertices with least distance, say, $d,$ in $G.$ Then, there exists two paths $P$ and $Q$ from $u$ to $v,$ where $\sigma(P(u,v))$ is positive and $\sigma(Q(u,v))$ is negative. We have to prove that $P$ and $Q$ are internally disjoint. If not, then there will be some vertices common to $P$ and $Q.$ Let $w_1,w_2,\dots, w_k$ be the common vertices in $P$ and $Q$ other than the end vertices and denote $u$ and $v$ as $w_0$ and $w_{k+1}$ respectively. Restrict the paths $P$ and $Q$ from $w_i$ to $w_{i+1}$ as $P_i$ and $Q_i$ for $i=0,1,2, \dots k.$ Then, $P=P_0 \cup P_1 \cup P_2 \cup \dots \cup P_k$ and $Q=Q_0\cup Q_1 \cup Q_2\cup \dots \cup Q_k $. All such $P_i$ and $Q_i$ should be of same order, otherwise if for some $i,$ $|P_i|<|Q_i|$ implies, through that $P_i,$ we can find a shortest path from $u$ to $v$ of length less than $k.$ Also, if all $P_i$ and $Q_i$ are of same sign, then $\sigma(P)=\prod_{i=0}^{k}\sigma( P_i)$ and $\sigma(Q)=\prod_{i=0}^{k}\sigma( Q_i)$ will be of same sign, a contradiction. Therefore, from the collection $\{(w_i,w_{i+1}): 0\leq i\leq k\},$ we can find atleast one pair $(w_j,w_{j+1}),$ having two path $P_j$ and $Q_j$ of same length, say, $l$ and opposite sign. If there is any path from $w_j$ to $w_{j+1}$ of length less than $l$ through that path we can find a path from $u$ to $v$ of length less than $k,$ a contradiction. Therefore, $P_j$ and $Q_j$ will form shortest paths from  $w_j$ to $w_{j+1}$ of length $l$ and of opposite sign. Thus,  $d_{\max}(w_j,w_{j+1})\ne d_{\min}(w_j,w_{j+1}).$ Therefore, $w_j$ and $w_{j+1}$ will be an incompatible pair in $\Sigma$ with distance less than $d,$ which is a contradiction. Hence, $P$ and $Q$ will be internally disjoint.
	
\end{proof}
\begin{thm}
	A $2$-connected non-geodetic signed graph is incompatible if and only if it has an even negative cycle $C_{2k}$ such that there exists two diametrically opposite vertices $u$ and $v$ on $C_{2k}$ in such a way that there are no shortest paths between $u$ and $v$ of length less than $k.$
\end{thm}
\begin{proof}
	Suppose that $\Sigma=(G,\sigma)$ is incompatible. Among all the pair of incompatible vertices in $\Sigma,$ let $u$ and $v$ be a pair with shortest distance, say, $k,$ in $G.$ Then, by Lemma \ref{L1} there exists two internally disjoint paths $P$ and $Q$ of order $k$ from $u$ to $v,$ and without loss of generality assume that $\sigma(P(u,v))$ is positive and $\sigma(Q(u,v))$ is negative. Therefore, the union of $P$ and $Q$ will be the cycle $C_{2k}$ in $\Sigma,$ infact a negative cycle with $u$ and $v$ as diametrically opposite vertices. 	
	
	Conversely, suppose that $\Sigma $ has an even negative cycle $C_{2k}$ such that there exists two diametrically opposite vertices $u$ and $v$ on $C_{2k}$ and there are no shortest path between $u$ and $v$ of length less than $k.$ Consider all the shortest paths from $u$ to $v$ of length $k.$ Since $u$ and $v$ are diametrically oposite,  there are two $uv$-paths of opposite sign and of length $k$ in $C_{2k}^-.$ Thus, $u$ and $v$ will be an incompatible pair of vertices in $\Sigma.$
\end{proof}

\section{Compatibility of the product of two signed graphs }
The following discussions mainly deal with the compatibility conditions and the distance matrix of some of the signed graph products for which we first recall the definitions of these products.
\begin{defn}\cite{scart}
	The cartesian product $\Sigma_1\times \Sigma_2$ of two signed graphs $\Sigma_1=(G_1,\sigma_1)$ and $\Sigma_2=(G_2,\sigma_2)$ is defined as the signed graph with vertex set and edge set are that of the cartesian product of the underlying unsigned graphs and the signature function $\sigma$ for the labeling of the edges is defined by\\
	$\sigma((u_i,v_j),(u_k,v_l))=
	\begin{cases}
	\sigma_1(u_i,u_k), \mbox{ if } j=l\\
	\sigma_2(v_j,v_l),  \mbox{ if } i=k
	\end{cases}
	$
	
\end{defn}

\begin{defn}\cite{scomp}
	The \textit{lexicographic product} $\Sigma_1[\Sigma_2]$ (also called composition) of two signed graphs $\Sigma_1=(V_1,E_1,\sigma_1)$ and $\Sigma_2=(V_2,E_2,\sigma_2)$ as the signed graph $(V_1\times V_2,E,\sigma)$ where the edge set is that of the lexicographic product of underlying unsigned graphs and the signature function $\sigma$ for the labeling of the edges is defined by\\

	$\sigma((u_i,v_j)(u_k,v_l))=
	\left\{
	\begin{array}{ll}
	\sigma_1(u_iu_k)  & \mbox{if } i\neq k \\
	\sigma_2(v_jv_l) & \mbox{if } i=k
	\end{array}
	\right.$ 
\end{defn}

The tensor product $G_1\otimes G_2$ of two graphs $G_1$ and $G_2$ is a graph with vertex set $V(G_1)\times V(G_2),$ in which two vertices $(u_1,u_2)$ and $(v_1,v_2)$ are adjacent if and only if $u_1$ is adjacent to $v_1$ in $G_1$ and $u_2$ is adjacent to $v_2$ in $G_2.$ The tensor product of two signed graphs is given in \cite{stens} as follows.

\begin{defn}\cite{stens}
	Let $\Sigma=(G,\sigma)$ be a signed graph, $\Sigma$ is called \textit{tensor product} of two signed graphs $\Sigma_1=(G_1,\sigma_1)$ and $\Sigma_2=(G_2,\sigma_2)$, i.e. $\Sigma=\Sigma_1\otimes \Sigma_2$ if $G=G_1\otimes G_2$ and for an edge $(u_1,u_2)(v_1,v_2)$ of $G,$ 
	
	$\sigma((u_1,u_2)(v_1,v_2))=\sigma_1(u_1,v_1)\sigma_2(u_2,v_2).$ 

\end{defn} 
Now we deal with the distance compatibility criterion of these products one by one. In the following theorems, the notation $d_{\Sigma}(u,v)$ for a compatible signed graph denotes $d_{\max}(u,v)=d_{\min}(u,v)$.
\begin{thm}\label{cart}
	The cartesian product  $\Sigma_1\times \Sigma_2$ is compatible if and only if $\Sigma_1$ and $\Sigma_2$ are compatible.
\end{thm}
\begin{proof}
	Suppose that $\Sigma_1\times \Sigma_2$ is compatible. Without loss of generality assume that $\Sigma_1$ is incompatible. Let $u_i$ and $u_j$ be an incompatible pair of vertices in $\Sigma_1.$ Also, $\Sigma_1$ being a subgraph of $\Sigma_1\times \Sigma_2,$ it is possible to find incompatible pairs of the form  $(u_i,v_k)$ and $(u_j,v_k)$ in $\Sigma_1\times \Sigma_2,$ a contradiction. Therefore, both $\Sigma_1$ and $\Sigma_2$ should be compatible.
	
	Conversely, suppose that $\Sigma_1$ and $\Sigma_2$ are compatible. Let $u=(u_1,u_2)$ and $v=(v_1,v_2)$ be any two vertices in $\Sigma_1\times \Sigma_2$ with $|d_{\Sigma_1}(u_1,v_1)|=d_1$ and $|d_{\Sigma_2}(u_2,v_2)|=d_2.$ Then, $|d_{\Sigma_1\times \Sigma_2}(u,v)|=d=d_1+d_2.$
	
	Let $u_1=x_0,x_1, \dots, x_{d_1}=v_1$ and $u_2=y_0,y_1,\dots, y_{d_2}=v_2$ be the shortest paths between $u_1$ and $v_1$ in $\Sigma_1$ and, $u_2$ and $v_2$ in $\Sigma_2$ respectively. Then, in any shortest path between $u$ and $v$ the first co-ordinate needs to travel $d_1$ steps and the second cordinate needs to travel $d_2$ steps.
	
	Let $S_1$ consists of all edges in $P$ that correspond to the $d_1$ steps which are part of $\Sigma_1$ and $S_2$ consists of all edges in $P$ that correspond to the $d_2$ steps which are part of $\Sigma_2$. Any shortest path between two vertices in $\Sigma_1\times \Sigma_2$ changes exactly one co-ordinate at any step. Then, for each $e\in S_1,$ the first co-ordinate of the end vertices will be of the form $x_i$ and $x_{i+1}$ and the second co-ordinate of the end vertices will be fixed. Similarlly, for each edge $e\in S_2,$ the second co-ordinate of the end vertices will be of the form $y_i$ and $y_{i+1}$ and the first co-ordinate of the end vertices will be fixed.
	Since, $S_1$ and $S_2$ contains all the edges in the shortest path from $u$ to $v,$ we can write $E(P)=S_1\cup S_2$. Thus, $\sigma(P(u,v))=\prod_{e\in S_1} \sigma(e) . \prod_{e\in S_2} \sigma(e).$
	
	Since, for each $e\in S_1,$ the second co-ordinate is fixed, $\sigma(e)$ will depend only on the first co-ordinate of the end vertices of $e.$\\
	Thus, 
	$\prod_{e\in S_1} \sigma(e)=\prod_{i=0}^{d_1-1}\sigma(x_ix_{i+1}).$
	
	That is, the product of the sign of all edges in $S_1$ will correspond to the sign of a shortest path connecting $x_0$ and $x_{d_1}.$ Since, $\Sigma_1$ is compatible the sign of all shortest paths from $x_0$ to $x_{d_1}$ will have unique sign, $\prod_{e\in S_1} \sigma(e)$ will be always unique.
	
	Similarly, $\prod_{e\in S_2} \sigma(e)=\prod_{i=0}^{d_2-1}\sigma(y_iy_{i+1}).$\\
	Since, $\Sigma_2$ is compatible, by using the above mentioned argument we can say that $\prod_{e\in S_2} \sigma(e)$ will be always unique.
	
	Thus, $\sigma(P(u,v))=\prod_{e\in S_1} \sigma(e) . \prod_{e\in S_2} \sigma(e)$ will be unique for any path $P$ connecting $u$ and $v.$
	
	That is, $u$ and $v$ are compatible pairs in $\Sigma_1\times \Sigma_2,$ for any two arbitrary vertices $u$ and $v.$ Hence, $\Sigma_1\times \Sigma_2$ is compatible.
	
\end{proof}

 \begin{thm}
	Let $\Sigma_1$ and $\Sigma_2$ be two signed graphs. Then $\Sigma_1[\Sigma_2]$ is compatible if and only if $\Sigma_1$ is compatible and $\Sigma_2$ is either all-positive or all-negative.
\end{thm}
\begin{proof}
	Suppose that $\Sigma_1[\Sigma_2]$ is compatible. As $\Sigma_1\times \Sigma_2$ is a subgraph of $\Sigma_1[\Sigma_2],$ by Theorem \ref{cart}, $\Sigma_1$ and $\Sigma_2$ are compatible. Suppose, if possible further that, $\Sigma_2$ contains at least one negative and one positive edge. Then, it is possible to find vertices $v_1, v_2$ and $v_3$ in $\Sigma_2$ such that the edges $v_1v_2$ and $v_2v_3$ are incident and having positive and negative sign respectively. Also, let $u_1u_2$ be an edge in $\Sigma_1.$ Consider the vertices $(u_1,v_1)$ and $(u_1,v_3)$ in $\Sigma_1[\Sigma_2].$ Then, consider the two shortest paths $P_1:(u_1,v_1)(u_1,v_2)(u_1,v_3)$ and $P_2:(u_1,v_1)(u_2,v_1)(u_1,v_3)$ from $(u_1,v_1)$ to $(u_1,v_3).$ The sign of these two paths are calculated as follows.\\
	$\sigma(P_1)=\sigma(v_1v_2)\sigma(v_2v_3),$  is negative.\\
	$\sigma(P_2)=\sigma(u_1u_2)\sigma(u_2u_1),$ is always positive.\\
	That is, $P_1$ and $P_2$ are two shortest paths from $(u_1,v_1)$ to $(u_1,v_3)$ in $\Sigma_1[\Sigma_2]$ with opposite signs. Hence, $(u_1,v_1)$ and $(u_1,v_3)$ form an incompatible pair in $\Sigma_1[\Sigma_2],$ a contradiction. Therefore, $\Sigma_2$ should be either all-positive or all-negative.
	
	Conversely, suppose that $\Sigma_1$ is compatible and $\Sigma_2$ is all-positive or all-negative.  Let $(u_i,v_k)$ and $(u_j,v_l)$ be any two vertices in $\Sigma_1[\Sigma_2].$ We have the following cases to deal with.
	
	\textbf{Case 1}: If $u_i\neq u_j$ 
	
Suppose that the shortest path $P$ from $(u_i,v_k)$ to $(u_j,v_l)$ has distance $d(u_i,u_j)=k.$ Then, $P$ should have $k$ edges and the sign of the edges will depend only on the sign of the first co-ordinates. Also, the first co-ordinates of an edge in $P$ will be adjacent vertices in $\Sigma_1.$ Therefore, sign of $P$ will be same as that of sign of the shortest path between $u_i$ and $u_j$ in $\Sigma_1.$ Since, $\Sigma_1$ is compatible, sign of any shortest path between $u_i$ and $u_j$ in $\Sigma_1$ will be unique. Therefore, sign of the shortest path between $(u_i,v_k)$ and $(u_j,v_l)$ in $\Sigma_1[\Sigma_2]$ is unique.

	\textbf{Case 2}: If $u_i=u_j$
	
	 Then, the shortest path between $(u_i,v_k)$ and $(u_i,v_l)$ moves either horizondally in $\Sigma_1[\Sigma_2],$ say, $P_1$ or the path which moves from $(u_i,v_k)$ to $(u_j,v_m),$ and then it moves to $(u_i,v_l),$ where $u_i \sim u_j$ $\Sigma_1,$ say, $P_2.$ In $P_1,$ since the first co-ordinate is fixed the sign of an edge will depends only on the second co-ordinates that are edges in $\Sigma_2.$ Since, the sign of the edges in $\Sigma_2$ are either all positive or all negative, and $d(v_l,v_k)=2,$ the sign of $P_1$ will be positive. Also, the sign of $P_2$ is 
	$\sigma(P_2)=\sigma(u_iu_j)\sigma(u_ju_i),$ always positive. Therefore, $(u_i,v_k)$ and $(u_i,v_l)$ are compatible. Hence, the converse follows.
	
\end{proof}
The following lemma gives necessary and sufficient condition for the connectedness of a tensor product of graphs.
	\begin{lem}[\cite{wch}]\label{ten}
	For connected graphs $G_1$ and $G_2,$ the tensor product graph $G_1 \otimes G_2$ is connected if and only if either $G_1$ or $G_2$ contains an odd cycle.	
\end{lem}
The formula for computing distance between two vertices in the  tensor product of two graphs is derived in \cite{tdist}. For a graph $G$ and two vertices $u,v \in V(G),$ the odd distance $od_G(u,v)$ is the length of the shortest odd walk joining $u$ and $v$ in $G,$ and the even distance $ed_G(u,v)$ is the length of the shortest even walk joining $u$ and $v$ in $G.$ If no walk of odd (even) length exists between $u$ and $v,$ then $od_G(u,v)= \infty $ ($ed_G(u,v)=\infty$).
\begin{lem}[\cite{tdist}]\label{2}
	Let $G$ and $H$ be two connected graphs, and let $u=(u_1,u_2), v=(v_1,v_2) \in V(G)\times V(H).$  Then,\\
	$d_{G\otimes H}(u,v)=\min\{\max\{od_G(u_1,v_1), od_H(u_2,v_2)\},\max\{ed_G(u_1,v_1),ed_H(u_2,v_2)\}\}.$ 
\end{lem}
Let $\Sigma_1$ and $\Sigma_2$ be two signed graphs. Taking cue from Lemma \ref{ten}, we need to assume that one of them contains an odd cycle for discussing the connected tensor product. In the following theorem, we get one way implication for the compatibility of the tensor product of $\Sigma_1$ and $\Sigma_2$ and we strongly believe that the converse is also true, the proof of which evades us for the time being.

\begin{thm}
	If the connected tensor product  $\Sigma_1\otimes \Sigma_2$ is compatible then, both $\Sigma_1$ and $\Sigma_2$ are compatible.
\end{thm}

\begin{proof}
	Suppose that the connected tensor product $\Sigma_1\otimes \Sigma_2$ is compatible. Without loss of generality assume that $\Sigma_1$ is incompatible. Let $u_0$ and $u_k$ be an incompatible pair of vertices in $\Sigma_1$ with distance $k.$ Then, there exists two shortest paths $P :u_0u_1u_2 \dots u_{k-1}u_k$ and $Q :u_0u._1u'_2 \dots u'_{k-1}u_k$ from $u_0$ to $u_k,$ where $\sigma(P)$ is positive and $\sigma(Q)$ is negative. Consider the following cases.

	\textbf{Case 1}: When $k$ is even.
	
	Then, $P$ will have even number of positive edges and even number of negative edges and $Q$ will have odd number of positive edges and odd number of negative edges.
	
	Let $vv'$ be an edge in $\Sigma_2.$ Consider the  following paths between the vertices $(u_0,v)$ and $(u_k,v)$ in $\Sigma_1\otimes \Sigma_2.$\\	
	$P' : (u_0,v) (u_1,v') \dots (u_{k-1},v') (u_k,v) \\ Q' : (u_0,v) (u'_1,v') \dots (u'_{k-1},v') (u_k,v).$
	
	These two paths are of order $k,$ and also by using Lemma \ref{2}, the shortest distance between $(u_0,v)$ and $(u_k,v)$ is	$\min\{\max\{l>k, 0\},\max\{k,0\}\}=k.$ Thus, $P'$ and $Q'$ will be two shortest paths from $(u_0,v)$ to $(u_k,v).$
	
	Then, $\sigma(P')=\sigma(u_0u_1)\sigma(vv')\sigma(u_1u_2)\sigma(vv') \dots \sigma(u_{k-2}u_{k-1}) \sigma(vv')\sigma(u_{k-1}u_k)\sigma(v'v)$ and $\sigma(Q')=\sigma(u_0u'_1)\sigma(vv')\sigma(u'_1u'_2)\sigma(vv') \dots \sigma(u'_{k-2}u'_{k-1}) \sigma(vv')\sigma(u'_{k-1}u'_k)\sigma(v'v).$
	
	If $\sigma(vv')$ is positive, then $\sigma(P')=\sigma(P)$ and $\sigma(Q')=\sigma(Q).$
	Also, if $\sigma(vv')$ is negative, then $\sigma(P')=\sigma(P)$ and $\sigma(Q')=\sigma(Q),$ since $\sigma(vv')$ will appear even number of times in each paths.
	
	\textbf{Case 2}: When $k$ is odd.
	
	Here, $P$ will have odd number of positive edges and even number of negative edges and $Q$ will have even number of positive edges and odd number of negative edges.
	
	Let $vv'$ be an edge in $\Sigma_2.$ Consider the following paths between the vertices $(u_0,v)$ and $(u_k,v')$ in $\Sigma_1\otimes \Sigma_2.$\\	
	$P'' : (u_0,v) (u_1,v') \dots (u_{k-1},v) (u_k,v') \\ Q'' : (u_0,v) (u'_1,v') \dots (u'_{k-1},v) (u_k,v').$
	
	These two paths are of order $k,$ and also by using Lemma \ref{2}, we get the shortest distance between $(u_0,v)$ and $(u_k,v')$ is $\min\{\max\{k, 1\},\max\{l>k,l'>1\}\}=k.$ Thus, $P''$ and $Q''$ are two shortest paths from $(u_0,v)$ to $(u_k,v').$
	
	Then, $\sigma(P'')=\sigma(u_0u_1)\sigma(vv')\sigma(u_1u_2)\sigma(vv') \dots \sigma(u_{k-2}u_{k-1}) \sigma(vv')\sigma(u_{k-1}u_k)\sigma(vv')$ and $\sigma(Q'')=\sigma(u_0u'_1)\sigma(vv')\sigma(u'_1u'_2)\sigma(vv') \dots \sigma(u'_{k-2}u'_{k-1}) \sigma(vv')\sigma(u'_{k-1}u'_k)\sigma(vv').$
	
	If $\sigma(vv')$ is positive, then $\sigma(P'')=\sigma(P)$ and $\sigma(Q'')=\sigma(Q).$
	If $\sigma(vv')$ is negative, then $\sigma(P'')=-\sigma(P)$ and $\sigma(Q'')=-\sigma(Q).$
	
	Therefore, corresponding to an incompatible pair in $\Sigma_1,$ we can find an incompatible pair in $\Sigma_1\otimes \Sigma_2,$ a contradiction to our assumption that $\Sigma_1\otimes \Sigma_2$ is compatible. Hence, $\Sigma_1$ and $\Sigma_2$ should be compatible.
	
\end{proof}

As mentioned above, we strongly believe that the converse of the above theorem is also true and we place it as a conjecture.
\begin{conj}
If  $\Sigma_1$ and $\Sigma_2$ are two compatible signed graphs, then the connected tensor product  $\Sigma_1\otimes \Sigma_2$ is compatible.
\end{conj}

\section{Distance matrices of compatible product of signed graphs}\label{section3}

To deal with the distance matrices of compatible product of signed graphs we use the Kronecker product of an $m\times n$ matrix $A=(a_{ij})$ and a $p\times q$ matrix $B$ which is defined to be the $mp\times nq$ matrix $A\otimes B=(a_{ij}B)$.

Let $\Sigma_1=(G_1,\sigma)$ and $\Sigma_2=(G_2,\sigma')$ be two compatible signed graphs with $|V(\Sigma_1)|=m$ and $|V(\Sigma_2)|=n$. Let $\sigma_{ij}$ and $\sigma'_{kl}$ be defined in $\Sigma_1$ and $\Sigma_2$ respectively, as follows. 
$$
\sigma_{ij} =
\begin{cases}
\sigma(P_{(u_i,u_j)}) &, \mbox{if } i\ne j ,\\
1 &, \mbox{ if } i=j\\

\end{cases}
$$

$$\sigma'_{kl} =
\begin{cases}
\sigma'(P_{(v_k,v_l)}) &, \mbox{if } k\ne l ,\\
1 &, \mbox{ if } k=l
\end{cases}
$$
Then,

	$ K^{D^\pm(\Sigma_1)}+I_m=\left(
	\begin{array}{ccccc}
	1 & \sigma_{12} & \dots & \sigma_{1m} \\
	\sigma_{21} & 1 & \dots & \sigma_{2m} \\
	\vdots & \vdots &\vdots & \vdots \\
	\vdots & \vdots &\vdots & \vdots \\
	\sigma_{m1} & \sigma_{m2} & \dots & 1 \\
	\end{array}
	\right)= \left(
\begin{array}{ccccc}
\sigma_{11} & \sigma_{12} & \dots & \sigma_{1m} \\
\sigma_{21} & \sigma_{22} & \dots & \sigma_{2m} \\
\vdots & \vdots &\vdots & \vdots \\
\vdots & \vdots &\vdots & \vdots \\
\sigma_{m1} & \sigma_{m2} & \dots & \sigma_{mm} \\
\end{array}
\right)$ 
	
Similarly,	$ K^{D^\pm(\Sigma_2)}+I_n=\left(
	\begin{array}{ccccc}
	1 & \sigma'_{12} & \dots & \sigma'_{1n} \\
	\sigma'_{21} & 1 & \dots & \sigma'_{2n} \\
	\vdots & \vdots &\vdots & \vdots \\
	\vdots & \vdots &\vdots & \vdots \\
	\sigma'_{n1} & \sigma'_{n2} & \dots & 1 \\
	\end{array}
	\right) =\left(
	\begin{array}{ccccc}
	\sigma'_{11} & \sigma'_{12} & \dots & \sigma'_{1n} \\
	\sigma'_{21} & \sigma'_{22} & \dots & \sigma'_{2n} \\
	\vdots & \vdots &\vdots & \vdots \\
	\vdots & \vdots &\vdots & \vdots \\
	\sigma'_{n1} & \sigma'_{n2} & \dots & \sigma'_{nn} \\
	\end{array}
	\right) $

\begin{thm}
	Let $\Sigma_1=(G_1,\sigma)$ and $\Sigma_2=(G_2,\sigma')$ be two compatible signed graphs where $|V(\Sigma_1)|=m$ and $|V(\Sigma_2)|=n.$ Then, the distance matrix of the cartesian product $\Sigma_1\times \Sigma_2$ is given by,\\
	$D(\Sigma_1\times \Sigma_2)=D(\Sigma_1)\otimes(K^{D^\pm(\Sigma_2)}+I_n)+(K^{D^\pm(\Sigma_1)}+I_m)\otimes D(\Sigma_2).$
\end{thm}

\begin{proof}
	Let $D(\Sigma_1)$ and $D(\Sigma_2)$ be the distance matrices of $\Sigma_1$ and $\Sigma_2$ respectively. Let $V(\Sigma_1)=\{u_1,u_2,\dots,u_m\}$ and $V(\Sigma_2)=\{v_1,v_2,\dots, v_n\}$. Suppose that $\Sigma_1\times \Sigma_2$ is compatible.  Let $\sigma_{ij}$ and $\sigma'_{kl}$ de defined in $\Sigma_1$ and $\Sigma_2$ respectively, as follows. 
	
	$$\sigma_{ij} =
	\begin{cases}
	\sigma(P_{(u_i,u_j)})  &, \mbox{if } i\ne j ,\\
	1 &, \mbox{ if } i=j\\
\end{cases}$$

$$\sigma'_{kl} =
\begin{cases}
\sigma'(P_{(v_k,v_l)})  &, \mbox{if } k\ne l ,\\
1 &, \mbox{ if } k=l
\end{cases}$$

Also the shortest path between two vertices $u_i$ and $u_j$ in $\Sigma_1$ and, $v_k$ and $v_l$ in $\Sigma_2$ are denoted by $d_{\Sigma_1}(i,j)$ and $d_{\Sigma_2}(k,l)$ respectively.
	Let $u=(u_i,u_j)$ and  $v=(v_k,v_l)$ be two vertices in $\Sigma_1\times\Sigma_2$. Then,
	\begin{equation*}
	d_{\Sigma_1\times \Sigma_2}(u,v)=\sigma_{ij}.\sigma'_{kl} (d(u_i,u_j)+d(v_k,v_l))
	=\sigma'_{kl}d_{\Sigma_1}(i,j)+\sigma_{ij}d_{\Sigma_2}(k,l)
	\end{equation*}
	
	Then, the distance matrix of $\Sigma_1\times\Sigma_2$ can be written in the form

	$D(\Sigma_1\times \Sigma_2) = \left(
	\begin{array}{ccccc}
	B_{1,1} & B_{1,2} & B_{1,3} & \dots & B_{1,m} \\
	B_{2,1} & B_{2,2} &B_{2,3} & \dots & B_{2,m} \\
	\vdots & \vdots & \vdots & \vdots & \vdots \\
	\vdots & \vdots & \vdots & \vdots & \vdots \\
	B_{m,1} & B_{m,2} & B_{m,3}  & \dots & B_{m,m} \\
	\end{array}
	\right)$\\
	
		where each block is $B_{i,j}$ is given by,
	
	$B_{i,j} =\\
	\left(
	\begin{array}{ccccc}
	\sigma'_{11}d_{\Sigma_1}(i,j)+\sigma_{ij}d_{\Sigma_2}(1,1) & \sigma'_{12}d_{\Sigma_1}(i,j)+\sigma_{ij}d_{\Sigma_2}(1,2) & \dots & \sigma'_{1n}d_{\Sigma_1}(i,j)+\sigma_{ij}d_{\Sigma_2}(1,n) \\
	\sigma'_{21}d_{\Sigma_1}(i,j)+\sigma_{ij}d_{\Sigma_2}(2,1) & \sigma'_{22}d_{\Sigma_1}(i,j)+\sigma_{ij}d_{\Sigma_2}(2,2) & \dots & \sigma'_{2n}d_{\Sigma_1}(i,j)+\sigma_{ij}d_{\Sigma_2}(2,n) \\
	\vdots & \vdots &\vdots & \vdots \\
	\vdots & \vdots &\vdots & \vdots \\
	\sigma'_{n1}d_{\Sigma_1}(i,j)+\sigma_{ij}d_{\Sigma_2}(n,1) & \sigma'_{n2}d_{\Sigma_1}(i,j)+\sigma_{ij}d_{\Sigma_2}(n,2) & \dots & \sigma'_{nn}d_{\Sigma_1}(i,j)+\sigma_{ij}d_{\Sigma_2}(n,n) \\
	\end{array}
	\right)$\\

	$B_{i,j}$ can be split into two matrices as $B_{i,j}'$ and $B_{i,j}'',$ given as

	$B_{i,j}' =
	\left(
	\begin{array}{ccccc}
	\sigma'_{11}d_{\Sigma_1}(i,j) & \sigma'_{12}d_{\Sigma_1}(i,j) & \dots & \sigma'_{1n}d_{\Sigma_1}(i,j) \\
	\sigma'_{21}d_{\Sigma_1}(i,j) & \sigma'_{22}d_{\Sigma_1}(i,j) & \dots & \sigma'_{2n}d_{\Sigma_1}(i,j) \\
	\vdots & \vdots &\vdots & \vdots \\
	\vdots & \vdots &\vdots & \vdots \\
	\sigma'_{n1}d_{\Sigma_1}(i,j) & \sigma'_{n2}d_{\Sigma_1}(i,j) & \dots & \sigma'_{nn}d_{\Sigma_1}(i,j) \\
	\end{array}
	\right)$\\

	That is, $B_{i,j}' =d_{\Sigma_1}(i,j)(K^{D^\pm(\Sigma_2)}+I_n)$
	
	$B_{i,j}'' =
	\left(
	\begin{array}{ccccc}
	\sigma_{ij}d_{\Sigma_2}(1,1) & \sigma_{ij}d_{\Sigma_2}(1,2) & \dots & \sigma_{ij}d_{\Sigma_2}(1,n) \\
	\sigma_{ij}d_{\Sigma_2}(2,1) & \sigma_{ij}d_{\Sigma_2}(2,2) & \dots & \sigma_{ij}d_{\Sigma_2}(2,n) \\
	\vdots & \vdots &\vdots & \vdots \\
	\vdots & \vdots &\vdots & \vdots \\
	\sigma_{ij}d_{\Sigma_2}(n,1) & \sigma_{ij}d_{\Sigma_2}(n,2) & \dots & \sigma_{ij}d_{\Sigma_2}(n,n) \\
	\end{array}
	\right)$\\

	That is, $B_{i,j}''=\sigma_{ij}(D(\Sigma_2))$

	Then, 	
	
	$ D(\Sigma_1\times \Sigma_2)=\left(
	\begin{array}{ccccc}
	d_{\Sigma_1}(1,1)(K^{D^\pm(\Sigma_2)}+I_n) & \dots  & d_{\Sigma_1}(1,m) (K^{D^\pm(\Sigma_2)}+I_n)\\
	d_{\Sigma_1}(2,1)(K^{D^\pm(\Sigma_2)}+I_n) & \dots  & d_{\Sigma_1}(2,m)(K^{D^\pm(\Sigma_2)}+I_n) \\
	\vdots & \vdots   &\vdots \\
	\vdots & \vdots  &\vdots\\
	d_{\Sigma_1}(m,1)(K^{D^\pm(\Sigma_2)}+I_n) & \dots  & d_{\Sigma_1}(m,m)(K^{D^\pm(\Sigma_2)}+I_n) \\
	\end{array}
	\right)$
	+

	$\left(
	\begin{array}{ccccc}
	\sigma_{11}D(\Sigma_2)  & \sigma_{12}D(\Sigma_2)& \dots & \sigma_{1n}D(\Sigma_2) \\
	\sigma_{21}D(\Sigma_2) & \sigma_{22}D(\Sigma_2)& \dots & \sigma_{2n}D(\Sigma_2) \\
	\vdots & \vdots &\vdots & \vdots \\
	\vdots & \vdots &\vdots & \vdots \\
	\sigma_{n1}D(\Sigma_2) & \sigma_{n2}D(\Sigma_2) & \dots & \sigma_{nn}D(\Sigma_2) \\
	\end{array}
	\right)$\\
	
	Thus, $D(\Sigma_1\times \Sigma_2)=D(\Sigma_1)\otimes(K^{D^\pm(\Sigma_2)}+I_n)+(K^{D^\pm(\Sigma_1)}+I_m)\otimes D(\Sigma_2).$
	
\end{proof}
\begin{thm}\label{lexmtx}
The distance matrix of the compatible lexicographic product of two signed graphs $\Sigma_1$ and $\Sigma_2$ is,

$D(\Sigma_1[\Sigma_2])=D(\Sigma_1)\otimes J_n
+I_m\otimes(2K^{{D^\pm}(\Sigma_2)}-A(\Sigma_2)).$
\end{thm}
\begin{proof}
	Let $\Sigma_1$ and $\Sigma_2$  be two signed graphs with $|V(\Sigma_1)|=m$ and $|V(\Sigma_2)|=n.$ Suppose that $\Sigma_1[\Sigma_2]$ is compatible. The distance between two vertices $u=(u_i,v_k) $ and $v=(u_j,v_l)$ in $\Sigma_1[\Sigma_2]$ will be as follows.

$$d_{\Sigma_1[\Sigma_2]}(u,v)=\begin{cases}
d_{\Sigma_1}(u_i,u_j),  & \mbox{if } u_i \ne u_j ,\\
1\sigma(v_k,v_l),& \mbox{if } u_i=u_j \mbox{ and } v_k \sim v_l \\
2\sigma(P_{(v_k,v_l)}),& \mbox{if } u_i=u_j \mbox{ and } v_k \nsim v_l \\
\end{cases}$$

The distance matrix of $\Sigma_1[\Sigma_2]$ can be written in the form

	$D= \left(
\begin{array}{ccccc}
B_{1,1} & B_{1,2} & B_{1,3} & \dots & B_{1,m} \\
B_{2,1} & B_{2,2} &B_{2,3} & \dots & B_{2,m} \\
\vdots & \vdots & \vdots & \vdots & \vdots \\
\vdots & \vdots & \vdots & \vdots & \vdots \\
B_{m,1} & B_{m,2} & B_{m,3}  & \dots & B_{m,m} \\
\end{array}
\right)$\\
with,

$B_{i,j}= \left(
\begin{array}{ccccc}
d((u_i,v_1),(u_j,v_1)) & d((u_i,v_1),(u_j,v_2)) & \dots & d((u_i,v_1),(u_j,v_n))\\
d((u_i,v_2),(u_j,v_1)) & d((u_i,v_2),(u_j,v_2)) & \dots & d((u_i,v_2),(u_j,v_n)) \\
\vdots & \vdots & \vdots \\
\vdots & \vdots & \vdots  \\
d((u_i,v_n),(u_j,v_1)) & d((u_i,v_n),(u_j,v_2)) & \dots & d((u_i,v_n),(u_j,v_n))\\
\end{array}
\right)$\\

Whenever $u_i=u_j,$ $$d((u_i,v_k),(u_i,v_l)) =\begin{cases}

1\sigma (v_k,v_l), & \mbox{ if } v_k \sim v_l \\
2\sigma (P_{(v_k,v_l)}), & \mbox{ if }  v_k \nsim v_l \\
0, & \mbox{ if }  v_k= v_l
\end{cases}$$

That implies, when $u_i=u_j,$ $B_{i,j}$ is nothing but $2K^{D^\pm(\Sigma_2)}-A(\Sigma_2).$ Thus, the diagonal blocks of $D$ will be $2K^{D^\pm(\Sigma_2)}-A(\Sigma_2).$

Also, whenever $u_i\ne u_j,$ $d((u_i,v_k),(u_j,v_l)) =
d_{\Sigma_1}(u_i,u_j).$

Then, $B_{i,j}= \left(
\begin{array}{ccccc}
d_{\Sigma_1}(u_i,u_j) & d_{\Sigma_1}(u_i,u_j) & \dots & d_{\Sigma_1}(u_i,u_j)\\
d_{\Sigma_1}(u_i,u_j) & d_{\Sigma_1}(u_i,u_j) & \dots & d_{\Sigma_1}(u_i,u_j) \\
\vdots & \vdots & \vdots & \vdots \\
\vdots & \vdots & \vdots & \vdots \\
d_{\Sigma_1}(u_i,u_j) & d_{\Sigma_1}(u_i,u_j) & \dots & d_{\Sigma_1}(u_i,u_j)\\
\end{array}
\right)$\\

Thus, the distance matrix of the compatible lexicographic product  $\Sigma_1[\Sigma_2]$ is
$D(\Sigma_1[\Sigma_2])=D(\Sigma_1)\otimes J_n
+I_m\otimes(2K^{{D^\pm}(\Sigma_2)}-A(\Sigma_2)).$

\end{proof}

\section{Distance spectra of some compatible signed graphs}
In this section we briefly discuss the distance spectra of some compatible signed graphs.

\begin{defn}
	Let $\Sigma$ be a compatible signed graph and $D(\Sigma)=(d_{ij})_{n\times n}$ be the distance matrix of $\Sigma,$ then the distance characteristic polynomial of $\Sigma$ is defined as $f(D(\Sigma),\lambda) = \det(\lambda I - D(\Sigma)),$ where $I$ is the identity matrix of order $n.$ 
	
	The roots of the characteristic equation $f(D(\Sigma),\lambda) = 0,$ denoted by $\lambda_1, \lambda_2, \dots, \lambda_n$ are called the distance eigenvalues of $\Sigma.$ If the distinct eigenvalues of $D(\Sigma)$ are  
	$\lambda_1 \geq \lambda_2 \geq \dots \geq \lambda_k$ and their multiplicities are $m_1, m_2, \dots, m_k,$ respectively, then the distance spectrum of $\Sigma$ is denoted by $\begin{pmatrix}
	\lambda_1  & \lambda_2 & \dots & \lambda_k\\
	m_1 & m_2 & \dots & m_k
	\end{pmatrix}.$
	
\end{defn}
The net-degree of a vertex $v$ in $\Sigma$ is $d^{\pm}_\Sigma(v)=d^+_\Sigma(v)-d^-_\Sigma(v)$ where, $d^+_\Sigma(v)$ is the number of positive edges and $d^-_\Sigma(v)$ is the number of negative edges incident with the vertex $v.$ A signed graph $\Sigma$ is said to be net-regular if every vertex has constant net-degree. The Petersen graph $+P$ is net-regular with net-degree $3$.

The distance-regular graphs with diameter $2$ are very special, and form a subject of their own. The well-known Petersen graph is a distance-regular graph with diameter $2$. 
While studying balance on signed Petersen graphs, T. Zaslavsky \cite{tzpet} proved that though there are $2^{15}$ ways to put signs on the edges of the Petersen graph $P$, in many respects only six of them are essentially different. He proved that 
\begin{thm}[\cite{tzpet}]
	There are precisely  six isomorphism types of minimal signed Petersen graph: $+P, P_1, P_{2,2}, P_{3,2}, P_{2,3}$, and $P_{3,3}$ . Each one is the unique minimal isomorphism type in its switching isomorphism class.
\end{thm}
Since, the shortest path between any two pair of vertices is unique (such graphs are called geodetic), signed Petersen graphs are always compatible. As Petersen graph is an important object in graph theory, we study the distance spectrum of these six signed Petersen graphs.  
Listed below are the distance characteristic polynomial of the six isomorphism types of minimal signed Petersen graph.

\begin{enumerate}
	\item $f(D^{\pm}(+P),\lambda)=\lambda^{10}-135\lambda^8-1080\lambda^7-3645\lambda^6-5832\lambda^5-3645\lambda^4.$
	\item $f(D^{\pm}(P_1),\lambda)=\lambda^{10}-135\lambda^8-504\lambda^7+2851\lambda^6+15688\lambda^5-5229\lambda^4-122256 \lambda^3-157680 \lambda^2.$
	\item $f(D^{\pm}(P_{2,2}),\lambda)=\lambda^{10}-135\lambda^8-216\lambda^7+5587\lambda^6+13648\lambda^5-77957\lambda^4-220888 \lambda^3+243912 \lambda^2+645984\lambda-308880.$
	\item 	$f(D^{\pm}(P_{2,3}),\lambda)=\lambda^{10}-135\lambda^8-184\lambda^7+6211\lambda^6+13720\lambda^5-111981\lambda^4-295840 \lambda^3+690800 \lambda^2+196800\lambda.$
	
	\item $f(D^{\pm}(P_{3,2}),\lambda)=\lambda^{10}-135\lambda^8+40\lambda^7+6675\lambda^6-4848\lambda^5-140725\lambda^4+195240 \lambda^3+986040 \lambda^2-2613600\lambda+1724976.$
	
	\item $f(D^{\pm}(P_{3,3}),\lambda)=\lambda^{10}-135\lambda^8-120\lambda^7+6435\lambda^6+6696\lambda^5-145725\lambda^4-126000 \lambda^3+1620000 \lambda^2+800000\lambda-7200000.$

\end{enumerate}

The graph with integral spectrum is of special interest in literature, as such, it is noticed that among the six signed Petersen graphs only  the all-positive Petersen graph $+P$ and $P_{3,3} \simeq -P$ have integral distance spectrum. Also, the eigenspace of $P$ and $-P$ corresponding to its distance eigenvalues are the same. The spectral values  of these two signed Petersen graphs are discussed below.

The distance matrix of the Petersen graph $+P$ can be represented as
$D^{\pm}(+P)=2J_{10}-2I_{10}-A(+P),$ where the adjacency spectrum of $+P$ is $\begin{pmatrix}
3  & 1 & -2\\
1 & 5 & 4
\end{pmatrix}.$ Hence, the distance spectrum of $+P$ is $\begin{pmatrix}
15  & 0 & -3\\
1 & 4 & 5
\end{pmatrix}.$

Also, the distance matrix of the Petersen graph $-P$ can be represented as
$D^{\pm}(-P)=2J_{10}-2I_{10}+3A(-P),$ where the adjacency spectrum of $-P$ is $\begin{pmatrix}
2 & -1 &-3	\\
4 & 5 & 1
\end{pmatrix}.$
The distance spectrum of $-P$ is $\begin{pmatrix}
9  & 4 &-5 \\
1 & 4 & 5
\end{pmatrix}.$

We end our discussion with a special case of lexicographic product $\Sigma[K_2^{\pm}]$ and compute its distance eigenvalues. First we require a preliminary lemma which is given below. 
\begin{lem}[\cite{kron}]
	If $A$ and $B$ are square matrix of order $m$ and $n$ respectively, then $A\otimes B$ is a square matrix of order $mn.$ Also, $(A\otimes B)(C\otimes D)=AC \otimes BD,$ if the products $AC$ and $BD$ exists.
\end{lem}

\begin{thm}
	Let $\Sigma_1=(G,\sigma_1)$ be a compatible signed graph  and $\Sigma_2=(K_2,\sigma_2)$. If the distance eigenvalues of $\Sigma_1$ are $\lambda_1\geq \lambda_2\geq \dots \geq\lambda_m$. Then, $D(\Sigma_1[\Sigma_2])$ has eigenvalues, \\
$(1)$	$2\lambda_i+1$, for $1\leq i\leq m$ (each of multiplicity one) and $-1$ (of multiplicity $m$), 
	if $K_2$ is positive.\\
$(2)$	$2\lambda_i-1$, for $1\leq i\leq m$ (each of multiplicity one) and $1$ (of multiplicity $m$), if $K_2$ is negative.
\end{thm}
\begin{proof}
	Let ${\bf{P}}_i$ be the eigenvector corresponding to the eigenvalue $\lambda_i$ of $D(\Sigma_1)$ for $1\leq i\leq m$. By Theorem \ref{lexmtx}, the distance matrix of $\Sigma_1[\Sigma_2]$ can be expressed as $D(\Sigma_1[\Sigma_2])=D(\Sigma_1)\otimes J_2
	+I_m\otimes D(\Sigma_2),$ where $J_2=J_{2\times 2}$. We deal with the following two cases.
	
		Case $(1)$ Suppose that $K_2$ is positive:\\
	Then, $1$ is an eigenvalue of $D(\Sigma_2)$ with eigenvector ${\bf j}$. Also, $J_2$ has eigenvalues $2$ with the eigenvector ${\bf j}$ and $0,$ each with multiplicity $1$. Then, 
	
	$D(\Sigma_1[\Sigma_2])({\bf{P}}_i\otimes {\bf{j}})=(D(\Sigma_1)\otimes J_2
	+I_m\otimes D(\Sigma_2))({\bf{P}}_i\otimes {\bf{j}})$

	= $D(\Sigma_1){\bf{P}}_i\otimes J_2{\bf{j}}+I_m{\bf{P}}_i\otimes D(\Sigma_2){\bf{j}}$
	= $\lambda_i{\bf{P}}_i\otimes 2{\bf{j}}+{{\bf{P}}_i\otimes} {\bf{j}}$
	
	=$(\lambda_i 2+1){(\bf{P_i\otimes j})}$\\
	That is, $2\lambda_i +1$ is an eigenvalue of $D(\Sigma_1[\Sigma_2])$ for $1\leq i\leq m.$
	
	Let $\bf{Q}$ be the eigenvector of $D(\Sigma_2)$ corresponding to the eigenvalue $-1.$ Then,
	
	$D(\Sigma_1[\Sigma_2])({\bf{P}}_i\otimes \bf{Q})=(D(\Sigma_1)\otimes J_2
	+I_m\otimes D(\Sigma_2))({\bf{P}}_i\otimes \bf{Q})$
	
	= $D(\Sigma_1){\bf{P}}_i\otimes J_2\bf{Q}+I_m{\bf{P}}_i\otimes D(\Sigma_2)\bf{Q}$
	= $\lambda_i{\bf{P}}_i\otimes 0\cdot{\bf{Q}}+{\bf{P}}_i\otimes( -1 )\bf{Q}$
	
	=$-1({\bf{P}}_i\otimes \bf{Q})$\\
	That is, $-1$ is an eigenvalue of $D(\Sigma_1[\Sigma_2])$ of multiplicity $m.$\\
	Case $(2)$ When $K_2$ is negative:\\
	Using the same proof as in $(1),$ above we can see that the eigenvalues of $D(\Sigma_1[\Sigma_2])$ are $2\lambda_i-1$, for $1\leq i\leq m$ each of multiplicity one and $1$ with multiplicity $m$.
	
\end{proof}

\section*{Acknowledgements}

The first and second authors would like to acknowledge their gratitude to the Council of Scientific and Industrial Research (CSIR), India, for the financial support under the CSIR Junior Research Fellowship scheme, vide order nos.: 09/1108(0032)/2018-EMR-I and 09/1108(0016)/2017-EMR-I, respectively. The fourth author would like to acknowledge her gratitude to Science and Engineering Research Board (SERB), Govt.\ of India, for the financial support under  the scheme Mathematical Research Impact Centric Support (MATRICS), vide order no.: File No. MTR/2017/000689.

 % Authors extend their thanks to the reviewer(s) for the critical comments and insightful suggestions which improved the content and presentation style of the content of the paper.

\end{document}